\documentclass [10pt]{amsart}
\usepackage{amsmath}
\usepackage{amssymb}
\usepackage{comment}
\usepackage{tikz-cd}

\newtheorem{thm}{Theorem}[section]
\newtheorem{cor}[thm]{Corollary}
\newtheorem{lem}[thm]{Lemma}

\theoremstyle{definition}

\theoremstyle{remark}
\newtheorem{rem}[thm]{Remark}
\numberwithin{equation}{section}

\newcommand{\beas}{\begin{eqnarray*}}
\newcommand{\eeas}{\end{eqnarray*}}
\newcommand{\bes} {\begin{equation*}}
\newcommand{\ees} {\end{equation*}}
\newcommand{\be} {\begin{equation}}
\newcommand{\ee} {\end{equation}}
\newcommand{\bea} {\begin{eqnarray}}
\newcommand{\eea} {\end{eqnarray}}
\newcommand{\ra} {\rightarrow}
\newcommand{\txt} {\textmd}

\newcommand{\R}{\mathbb R}

\newcommand{\T}{\mathbb T}
\newcommand{\N}{\mathbb N}

\newcommand{\g}{\mathfrak{g}}
\begin{document}

\title[\tiny{Improved Ingham-type result on $\R^d$ and on nilpotent Lie Groups}]{Improved Ingham-type result on $\R^d$ and on connected, simply connected nilpotent Lie Groups}

\author{\tiny{Mithun Bhowmik}}

\address{Stat-Math Unit, Indian Statistical Institute, 203 B. T. Road, Kolkata - 700108,  India.}

\email{mithunbhowmik123@gmail.com}

\thanks{This work was supported by Indian Statistical Institute, India (Research fellowship).}


\begin{abstract}
In \cite{BRS} we have characterized the existance of a non zero function vanishing on an open set in terms of the decay of it's Fourier transform on the $d$-dimensional Euclidean space, the $d$-dimensional torus and on connected, simply connected two step nilpotent Lie groups. In this paper we improved these results on $\R^d$ and prove analogus results on connected, simply connected nilpotent Lie groups.
\end{abstract}

\subjclass[2010]{Primary 22E25; Secondary 22E30, 43A80}

\keywords{nilpotent Lie group}
\maketitle

\section{Introduction}
It is a well known fact in harmonic analysis that if the Fourier transform of an integrable function on the real line is very rapidly decreasing then the function can not vanish on a nonempty open set unless it vanishes identically. For instance if, let $f\in L^1(\R)$ and $a>0$ be such that its Fourier transform satisfies the estimate
\bes
|\widehat f(\xi)| \leq Ce^{-a|\xi|}, \:\:\:\: \txt{ for all }\xi \in \R.
\ees 
If $f$ vanishes on a nonempty open set then $f$ is identically zero. This initial observation motivates one to endeavour for a more optimal decay of the Fourier transform $\widehat{f}$ for such a conclusion to hold. For instance we may ask:  if $\widehat{f}$ decays faster than $1/(1+|\cdot|)^n$ for all $n \in \N$ but slower than the function  $e^{-a|\cdot|}$, can $f$ vanish on a nonempty open set without being identically zero? A more precise question  could be:  is there a nonzero integrable function $f$ on $\R$ vanishing on a nonempty open set such that its Fourier transform satisfies the following estimate 
\be \label{introdecay} 
|\widehat f(\xi)| \leq Ce^{-\frac{|\xi|}{\log|\xi|}}, \:\:\:\: \txt{ for large } |\xi|?
\ee
The answer to the above question is in the negative and follows from a classical result due to Ingham \cite{I}. Analogous results were also obtained by Paley-Wiener, Levinson and Hirschman \cite{PW1, PW, L1, Hi}. In \cite{BRS}, we have obtained analogues results on $d$-dimensioan Euclidean spaces, $d$-dimensional torus and on connected, simply connected two step nilpotent Lie groups. In the context of Euclidean space we have the following theorem \cite{BRS}. For $f\in L^1(\R^d)$, we shall define its Fourier transform $\widehat f$ by
\bes
\widehat f(\xi)=\int_{\R^d}f(x)~ e^{-2\pi i x\cdot \xi}~ dx,\:\:\:\: \txt{ for } \xi\in\R^d.
\ees
For $1 < p \leq 2$, the definition of Fourier transform of $f \in L^p(\R^d)$ is extended in the usual way. 
\begin{thm} \label{introthm}
Let $f\in L^p(\R^d)$, for some $p$ with $p\in [1, 2]$ and $\psi: [0,\infty) \ra [0,\infty)$ be a locally integrable function. Suppose that the Fourier transform $\widehat f$ of $f$ satisfies the following estimate
\be \label{introest}
|\widehat f(\xi)| \leq  C P(\xi) e^{-\psi(\|\xi\|)},
\ee
for almost every $\xi \in \R^d$, where $P(\xi)$ is a polynomial and we define 
\bes
I = \int_1^{\infty} \frac{\psi(t)}{t^2} dt.
\ees

\begin{enumerate}
\item[(a)] Let $\psi(t) = \theta(t)t,$ for $t \in [0, \infty)$, where $\theta : [0,\infty) \ra [0,\infty)$ is a decreasing function which decreases to zero as $t$ goes to infinity. If $f$ vanishes on any nonempty open set in $\R^d$ and $I$ is infinite then $f$ is zero on $\R^d$. Conversely, if $I$ is finite, then given any positive real number $l$ there exists a nonzero even function $f \in C_c(\R^d)$ supported in ball $B(0,l)$ of radius $l$, centered at zero satisfying (\ref{introest}).
\item[(b)] Let $\psi$ be an increasing function on $[0,\infty)$. If $f$ vanishes on any nonempty open set in $\R^d$ and $I$ is infinite then $f$ is zero on $\R^ d$. Conversely, if $I$ is finite, given any positive real number $l$ there exists a nonzero $f \in C_c(\R^d)$ supported in $B(0,l)$ satisfying (\ref{introest}).
\item[(c)] If
\bes
\txt{supp }f \subseteq \{x\in \R^d~|~ x\cdot \eta \leq s\}, 
\ees
and $I$ is infinite then $f$ is zero on $\R^d$. Conversely, if $I$ is finite then  there exists a nonzero $f \in L^2(\R^d)$ vanishing on the set $\{x \in \R^d : x\cdot \eta\geq x_0 \}$, for some $x_0 \in \R$ and $\eta \in S^{d-1}$ satisfying (\ref{introest}).
\end{enumerate}
\end{thm} 
In \cite{BRS} we have shown that the original statements of the results of Ingham, Levinson and Paley-Wiener on $\R$ are follows form the theorem above.
We shall show in this paper that it is possible to prove analogues of Theorem \ref{introthm} under the following variant of (\ref{introest})
\be\label{modiintroest}
\int_{\R^d}\frac{|\widehat f(\xi)|^qe^{q\psi(\|\xi\|)}}{(1+\|\xi\|)^N} ~ d\xi<\infty,
\ee
where $q\in [1,\infty]$ (with obvious modification for $q=\infty$) and $N$ is a positive real number. Our next main result in this paper is an analogue of Theorem \ref{introthm}, c) on the connected, simply connected nilpotent Lie group $G$. 

We shall use the following standard notation in this paper: $\txt{supp}~f$ denotes the support of the function $f$ and $C$ denotes a constant whose value may vary. For a finite set $A$ we shall use the symbol $\#A$ to denote the number of elements in $A$. $B(0,r)$ denotes the open ball of radius $r$ centered at $0$ in $\R^d$ and $\bar{B}(0,r)$ denotes its closure. For $x,y \in \R^d$, we shall use $\|x\|$ to denote the norm of the vector $x$ and $x \cdot y$ to denote the Euclidean inner product of the vectors $x$ and $y$. A function $\psi : \R^d \ra [0,\infty)$ is said to be radially increasing (or radially decreasing) if $\psi$ is radial and satisfies the condition $\psi(x)\geq \psi(y)$ (or $\psi(x)\leq \psi(y)$) whenever $|x|\geq |y|$. We shall often consider a radial function on $\R^d$ as an even function on $\R$ or equivalently, as a function on $[0,\infty)$.




\section{Results of Ingham, Levinson and Paley-Wiener on $\R^d$}

In this section our aim is to improve Theorem \ref{introthm} in the following way. Instead of the pointwide decay (\ref{introest}) on the Fourier transform side we will consider the weighted $L^q$ estimate for $1\leq q\leq \infty$ of the form (\ref{modiintroest}). Using convolution techniques we will be able to show that it is enough to prove the result for $q=1$ and $N=0$ and rest of the proof goes like the proof of Theorem \ref{introthm}. First we state and prove an analogue of Theorem \ref{introthm} (a) on $\R^d$.

\begin{thm}\label{ingrn}
Let $\theta:\R^d\rightarrow [0,\infty)$ be a radially decreasing measurable function with $\lim_{\|\xi\|\to\infty}\theta(\xi)=0$ and
\be \label{idefn}
I=\int_{\|\xi\|\geq 1}\frac{\theta (\xi)}{\|\xi\|^d} ~ d\xi.
\ee
\begin{enumerate}
\item[(a)] Let $f\in L^p(\R^d)$, $p\in [1,2]$, be such that
\be \label{ingcond}
\int_{\R^d} \frac{|\widehat{f}(\xi)|^q ~ e^{q \theta (\xi) \|\xi\|}}{(1+ \|\xi\|)^N} ~ d\xi < \infty, 
\ee
for some $q \in [1, \infty)$ and some $N\geq 0$ or 
\be \label{ingcondinfty}
|\widehat f(\xi)| \leq C ~ |P(\xi)| ~ e^{-\theta(\xi)\|\xi\|}, \:\:\:\: \txt{ for almost every } \xi\in \R^d,
\ee
where $P$ is a polynomial. If $f$ vanishes on a nonempty open subset in $\R^d$ and $I$ is infinite, then $f$ is zero almost everywhere on $\R^d$. 
\item[(b)] If $I$ is finite then there exists a nontrivial $f\in C_c^{\infty}(\R^d)$ satisfying the estimate (\ref{ingcond}) (or (\ref{ingcondinfty}) if $q=\infty$).
\end{enumerate}
\end{thm}

\begin{proof}
We shall first prove the assertion (b). 
For $q= \infty$ it follows from Theorem \ref{introthm}, (a). For $q\in [1, \infty)$ we choose $\phi_1 \in C_c^\infty(\R^d)$ with $\txt{supp } \phi_1 \subseteq B(0, l/2)$ and consider the function $f = f_0* \phi_1$. Clearly, support of the function $f$ is contained in $\overline{B(0, l)}$ and 
\bes
\int_{\R^d} \frac{|\widehat f(\xi)|^q ~ e^{q \theta(\xi)\|\xi\|}}{(1+\|\xi\|)^N}~ d\xi \leq  C \int_{\R^d} |\widehat {\phi_1}(\xi)|^q ~ d\xi < \infty.
\ees 
This, in particular, proves (b). It now remains to prove (a).
 
For part (a), we first show that it suffices to prove the case $q=1, N=0$. Suppose $f\in L^p(\R^d)$ vanishes on an open set $U\subseteq \R^d$ and satisfies (\ref{ingcond}), for some $q> 1$ and $N\in \N$. Since the condition (\ref{ingcond}) is invariant under translation of $f$ we can assume that $f$ vanishes on $B(0, l)$, for some positive $l$. If we choose $\phi \in C_c^\infty(\R^d)$ supported in $B(0, l/2)$ then $f*\phi$ vanishes on the ball $B(0, l/2)$. In fact, if $\|x\|< l/2$ then
\bes
\int_{\R^d} f(x-y) \phi(y) ~ dy = \int_{B\left(0,\frac{l}{2}\right)} f(x-y) \phi(y) ~ dy = 0,
\ees
as 
\bes
\|x-y\| \leq \|x\| + \|y\| < l.
\ees
By using H\"older's inequality we get
\beas
&& \int_{\R^d}|\widehat{(f*\phi)}(\xi)| ~ e^{\theta(\xi)\|\xi\|}~ d\xi \\
&\leq & \left( \int_{\R^d} \frac{|\widehat f(\xi)|^q ~ e^{q \theta(\xi)\|\xi\|}}{(1+\|\xi\|)^N} \right)^{\frac{1}{q}} \|(1+\|\xi\|)^N \widehat \phi(\xi)\|_{L^{q'}(\R^d)}\\
& < & \infty,
\eeas
as $N/ q$ is smaller than $N$. Here $q'$ satisfies the relation 
\bes
\frac{1}{q}+ \frac{1}{q'} = 1.
\ees
Hence, by the case $q=1, N=0$ it follows that $f*\phi$ vanishes identically. As $\phi \in C_c^\infty(\R^d)$ we have that $\widehat \phi$ is nonzero almost everywhere. This implies that $\widehat f$ vanishes almost everywhere and so does $f$. The same technique can be applied to reduce the case $q=1$ and $N\in \N$ to the case $q=1$ and $N=0$ by using H\"older's inequality. For the case $q= \infty$, we get from (\ref{ingcondinfty}) that 
\bes
\int_{\R^d}|\widehat{f*\phi}(\xi)|~ e^{\theta(\xi)\|\xi\|} ~ d\xi \leq \int_{\R^d} ~ |P(\xi)| ~ |\widehat \phi(\xi)| ~ d\xi < \infty.
\ees 
So, without loss of generality, we assume that $f \in L^p(\R^d
)$ such that $\widehat f$ satisfies the condition
\be \label{ingsplcond}
\int_{\R^d} |\widehat f(\xi)| ~ e^{\theta(\xi) \|\xi\|} ~ d\xi = C' < \infty.
\ee
Rest of the proof goes similarly as the proof of   Theorem 2.2 in \cite{BRS}.
\end{proof}

As in Theorem \ref{ingrn} we will have the following version of Theorem \ref{introthm}, (b) on $\R^d$.
\begin{thm}\label{levrn}
Let $\psi:\R^d \rightarrow [0,\infty)$ be a radially increasing function and
\bes
I=\int_{\|\xi\|\geq 1} \frac{\psi(\xi)}{\|\xi\|^{d+1}} ~ d\xi.
\ees

\begin{enumerate}
\item[(a)] Let $f\in L^p(\R^d)$, $p\in [1,2]$, be a function satisfying the estimate
\be \label{levcondp}
\int_{\R^d} \frac{|\widehat{f}(\xi)|^q ~ e^{q \psi(\xi)}}{(1+ \|\xi\|)^N} ~ d\xi < \infty, 
\ee
for some $q \in [1, \infty)$ and some $N\geq 0$ or 
\be \label{levcond}
|\widehat f(\xi)| \leq C ~ |P(\xi)| ~ e^{-\psi(\xi)}, \:\:\:\: \txt{ for almost every } \xi\in \R^d,
\ee
where $P$ is a polynomial. If $f$ vanishes on a nonempty open subset in $\R^d$ and $I$ is infinite, then $f$ is zero almost everywhere on $\R^d$. 
\item[(b)] If $I$ is finite then there exists a nontrivial $f\in C_c^{\infty}(\R^d)$ satisfying the estimate (\ref{levcondp}) (or (\ref{levcond}) if $q=\infty$).
\end{enumerate}
\end{thm}

In the following, we shall improved Theorem \ref{introthm}, c) which will be used in the next section to prove an analogous result for connected simply connected nilpotent Lie groups (see Theorem \ref{nilpotent}).

\begin{thm} \label{paleywiener}
Let $\psi: [0,\infty)\ra [0, \infty)$ be a locally integrable function and 
\bes
I=\int_1^{\infty}\frac{\psi (t)}{t^2}dt.
\ees 
\begin{enumerate}
\item[(a)] Let $f \in L^p(\R^d)$, $p\in [1, 2]$, be such that \bes
\txt{supp }f \subseteq \{x\in \R^d~|~ x\cdot \eta \leq s\}, 
\ees
for some $\eta\in S^{d-1}$ and $s\in \R$ and $\widehat f$ satisfies the estimate 
\be \label{pwqest}
\int_{\R^d} \frac{|\widehat{f}(\xi)|^q ~ e^{q \psi(|\xi\cdot \eta|)}}{(1+ \|\xi\|)^N} ~ d\xi < \infty.
\ee 
for some $q\in [1, \infty)$ and some $N\geq 0$, or
\be \label{pwinftyest}
|\widehat f(\xi)| \leq C ~ (1+\|\xi\|)^N ~ e^{-\psi(\xi \cdot \eta)}, \:\:\:\: \txt{ for almost every } \xi \in \R^d. 
\ee
If the integral $I$ is infinite then $f$ is the zero function.

\item[(b)] If $\psi$ is non-decreasing and $I$ is finite then there exists a nontrivial $f\in C_c(\R^d)$ satisfying the estimate (\ref{pwqest}), for some $q\in [1, \infty)$ and all $\eta\in S^{d-1}$ or (\ref{pwinftyest}), for $q= \infty$ and all $\eta\in S^{d-1}$.
\end{enumerate}
\end{thm}

\begin{proof}
We shall first prove a) for the case $p=2$. Then for all $p \in [1,2)$ the result will follow by reducing it to the case $p=2$ as was done in the proof of Theorem \ref{introthm}, (c). As in the proof of Theorem \ref{ingrn} it suffices to prove the case  $q=1$ and $N=0$. Now, by using translation and rotation of the function $f$, we can assume without loss of generality that $\eta= e_1= (1, 0, \cdots, 0)$ and $s=0$. Then, by writing $\xi=(\xi_1, \xi_2, \cdots, \xi_d)$ the hypothesis (\ref{pwqest}) becomes 
\be \label{reducedhyp}
\int_{\R^d} |\widehat{f}(\xi)| ~ e^{\psi(|\xi_1|)} ~ d\xi < \infty.
\ee
For $y \in \R^{d-1}$, we now define $g_y$ by
\bes
g_{y}(x)= \mathcal{F}_{d-1} f(x, y), \:\:\:\: \txt{ for almost every } x\in \R.
\ees It then follows that for almost every $y\in \R^{d-1}$, $g_y\in L^2(\R)$ with 
\bes
\txt{supp }g_y \subseteq \{x\in \R ~|~ x\leq 0\},
\ees 
and by (\ref{reducedhyp})
\be \label{lpcond}
\int_{\R} |\widehat {g_y}(t)| ~ e^{\psi(|t|)} ~ dt < \infty.
\ee
As $y$ varies over a set of full $(d-1)$ dimensional Lebesgue measure, we just need to prove that $g_y$ is the zero function. By the Paly-Wiener theorem (Theorem 2.7, \cite{BRS}) it suffices to show that 
\bes
\int_{\R}{\frac{|\log(|\widehat{g_y}(t)|)|}{1+t^2}dt} = \infty.
\ees
If 
\be \label{gyeintfinite}
\int_\R \frac{|\log(|\widehat{g_y}(t)| e^{\psi(|t|)})|}{1+t^2} ~ dt < \infty 
\ee
then
\beas
\int_{\R}{\frac{|\log(|\widehat{g_y}(t)|)|}{1+t^2}dt} &=& \int_{\R}{\frac{|\log(|\widehat{g_y}(t)|e^{\psi(|t|)})- \psi(|t|)|}{1+t^2}dt}\\
&\geq& \int_{\R}{\frac{\psi(|t|)}{1+t^2}dt}- \int_{\R}{\frac{|\log(|\widehat{g_y}(t)| e^{\psi(|t|)})|}{1+t^2}dt}.
\eeas
As $I$ is infinite, it follows from (\ref{gyeintfinite}) that
\bes
\int_{\R} \frac{|\log(|\widehat{g_y}(t)|)|}{1+t^2} ~ dt
\ees
is divergent. Hence, by the Paly-Wiener theorem (Theorem 2.7, \cite{BRS}) it follows that $g_y$ is the zero function. Now, suppose 
\be \label{gyeintinfinite}
\int_\R \frac{|\log(|\widehat{g_y}(t)| e^{\psi(|t|)})|}{1+t^2} ~ dt = \infty .
\ee
For a measurable function $F$ on $\R^d$, we define
\beas
\log^+|F(x)|= \max\{\log|F(x)|, 0\} \\
\log^-|F(x)|= - \min\{\log|F(x)|, 0\},
\eeas
and hence 
\bes
\vline ~\log|F(x)|~\vline = \log^+|F(x)| + \log^-|F(x)|.
\ees
As $\log^+|F(x)|$ is always smaller than $|F(x)|$ we get that
\bes 
\int_{\R}{\frac{\log^+ \left(|\widehat{g_y}(t)| e^{\psi(|t|)}\right)}{1+t^2}dt}\leq \int_{\R}{\frac{|\widehat{g_y}(t)|e^{\psi(|t|)}}{1+t^2}dt} < \infty,
\ees
by (\ref{lpcond}).  From (\ref{gyeintinfinite}) we now conclude that
\bes 
\int_\R \frac{\log^-(|\widehat{g_y}(t)| e^{\psi(|t|)})}{1+t^2} ~ dt = \infty .
\ees
But
\beas
\int_{\R}{\frac{\log^-\left(|\widehat{g_y}(t)| e^{\psi(|t|)}\right)}{1+t^2}dt} &= & \int_{\big\{t ~\vline~ |\widehat{g_y}(t)|e^{\psi(|t|)}\leq 1\big\}}{\frac{\log^-\left(|\widehat{g_y}(t)| e^{\psi(|t|)}\right)}{1+t^2}dt}\\
&\leq& \int_{\R}{\frac{\log^-|\widehat{g_y}(t)|}{1+t^2}dt},
\eeas
as on the set $\{t ~\vline~ |\widehat{g_y}(t)|e^{\psi(|t|)}\leq 1\}$ we have 
\bes
|\widehat{g_y}(t)| \leq |\widehat{g_y}(t)|e^{\psi(|t|)} \leq 1,
\ees and hence
\bes
\log^-\left(|\widehat{g_y}(t)|\right) \geq \log^-\left(|\widehat{g_y}(t)|e^{\psi(|t|)}\right).
\ees
Therefore the integral
\bes
\int_{\R}{\frac{\log^-\left(|\widehat{g_y}(t)|\right)}{1+t^2} ~ dt}
\ees
is divergent. Hence, the integral
\bes
\int_{\R}{\frac{|\log\left(|\widehat{g_y}(t)|\right)|}{1+t^2} ~ dt}
\ees
is divergent. By the Paly-Wiener theorem (Theorem 2.7, \cite{BRS}) it now follows that $g_y$ is the zero function. This completes the proof of part a). To prove Part b) we observe that since $\psi$ is non decreasing
\bes
\psi(|\xi \cdot \eta|) \leq \psi(\|\xi\|), \:\:\:\: \txt{ for } \xi \in \R^d.
\ees 
Therefore the proof follows from Theorem \ref{levrn}, (b).
\end{proof}

\begin{rem} \em
\begin{enumerate}
\item It is clear from the statement of Theorem \ref{paleywiener} that if $f$ is compactly supported continuous function then $\eta\in S^{d-1}$ can be taken to be arbitrary.
\item If $\psi$ is assumed to be a radial increasing function then Theorem \ref{introthm}, (c) follows from the case $p=\infty$ of Theorem \ref{paleywiener} and also from Theorem \ref{levrn}.
\end{enumerate}
\end{rem}

\section{Nilpotent Lie Groups}
In this section, our aim is to prove an analogue of Theorem \ref{paleywiener} in the context of connected, simply connected nilpotent Lie groups. 

\subsection{Preliminaries on Nilpotent Lie Groups}
In this section, we shall discuss the preliminaries and notation related to  connected, simply connected nilpotent Lie groups. These are standard and can be found, for instance, in \cite{BB, BTH, CG}.

Let $G$ be a connected, simply connected nilpotent Lie group with Lie algebra $\g$ and let $\g^*$ be the vector space of real-valued linear functionals on $\g$. In this case, the exponential map $\exp : \g \ra G$ becomes an analytic diffeomorphism, which enables us to identify $G$ with $\g\cong \R^d$, where $d= \dim \g$. Let 
\be \label{JHB}
\{0\}=\g_0 \subset \g_1 \subset \cdots \subset \g_d=\g
\ee
be a Jordan-H\"older series of ideals of the nilpotent Lie algebra $\g$ such that $\dim \g_j= j$, for $j=0, \cdots, d$ and ad$(X)\g_{j}\subseteq \g_{j-1}$, for $j=1, \cdots, d$ and for all $X\in \g$ (\cite{CG}, Theorem 1.1.9). We choose from this sequence a Jordan-H\"older basis 
\be \label{JHB}
\{X_1, \cdots, X_d\}, \:\:\:\: \txt{ where } X_j\in \g_j \backslash \g_{j-1},
\ee  
for $j=1, \cdots, d$. In particular, $\mathbb{R}X_1$ is contained in the center of $\g$. Let $\{X_1^*,\cdots, X_d^*\}$ be the basis of $\g^*$  dual to $\{X_1,\cdots, X_d\}$. We define coordinates on $G$ by 
\bes
(x_1, \cdots, x_d) \cong \exp(x_1 X_1 + \cdots + x_d X_d ). 
\ees
By the Baker-Campbell-Hausdorff formula $\exp(X) \exp(Y ) = \exp(Z)$, where
\be \label{BCHformula}
Z = X+Y + \frac{1}{2}[X,Y] + \frac{1}{12}[X,[X,Y]] - \frac{1}{12}[Y,[X,Y]]+\cdots
\ee
(\cite{CG}, P.12), which is a finite sum as $G$ is a nilpotent Lie group. We introduce a `norm function' on $G$ defined by 
\bes
\|x\|=\left(x_1^2+\cdots+x_d^2\right)^{\frac{1}{2}}, \:\:\:\: x=\exp(x_1X_1+\cdots+x_dX_d)\in G, ~  x_j\in \mathbb{R}.
\ees
The 
composite map 
\beas
&& \mathbb{R}^d \longrightarrow  \mathfrak{g}\longrightarrow G, \\
&& (x_1,\cdots,x_d) \mapsto \sum_{j=1}^d{x_jX_j}\mapsto
\exp(\sum_{j=1}^d{x_jX_j}),
\eeas
is a diffeomorphism and maps the Lebesgue measure on $\R^d$ to a Haar measure on $G$. We shall also identify 
$\g^*$ with $\R^d$ via the map 
\bes
\nu=(\nu_1, \cdots, \nu_d)\rightarrow \sum_{j=1}^d{\nu_jX_j^*}, \:\:\:\: \nu_j\in \R,
\ees
where $\{X_1^*, \cdots, X_d^*\}$ is the dual basis of $\g^*$. On $\g^*$ we define the Euclidean norm relative to this basis by 
\bes
\|\sum_{j=1}^d{\nu_jX_j^*}\|=\left(\nu_1^2+\cdots + \nu_d^2\right)^{\frac{1}{2}}=\|\nu\|.
\ees

For $\nu\in \g^*$ we have a natural bilinear map \bes
(X, Y) \mapsto B_\nu(X, Y)= \nu([X, Y]).
\ees
This map is antisymmetric and the radical of this map is 
\bes
\{X\in \g ~|~ B_\nu(X, Y)= 0, \:\: \txt{  for all } Y\in \g \} = r_\nu.
\ees 
Clearly, $B_\nu$ descends to nondegenerate, antisymmetric bilinear map $\tilde B_\nu : \g/ r_\nu \times \g/ r_\nu \ra \R$. Hence, $r_\nu$ is of even codimension in $\g$ (\cite{CG}, Lemma 1.3.2). Since the coadjoint orbit $O_\nu$ of $\nu$ is diffeomorphic to $G/R_\nu$, it follows that $O_\nu$ is even dimensional. An index $j\in\{1, 2, \cdots, d\}$ is called a jump index for $\nu\in \g^*$ if
\bes
r_\nu +\g_j\neq r_\nu + \g_{j-1},
\ees
where the Jordan-H\"older sequence $\g_j$ are as in (\ref{JHB}) (see \cite{BB}). Let $e(\nu)$ be the set of jump indices for $\nu\in \g^*$. This set contains $\dim(O_\nu)$ indices, which is an even number. Moreover, there are disjoint sets of indices $P$ and $Q$ such that $P\cup Q= \{1, \cdots, d\}$ and a $G$-invariant nonempty Zariski open set $\mathcal U \subseteq \g^*$ such that $e(\nu)= P$, for all $\nu\in \mathcal{U}$ (\cite{CG}, Theorem 3.1.2; \cite{BB}). The elements of $\mathcal{U}$ are called generic linear functionals. For $\nu\in \mathcal{U}$, we define 
\bes
B_{\nu, P}=\nu([X_i,X_j])_{i,j\in P}.
\ees 
Then the Pfaffian $Pf(\nu)$ is given by 
\be \label{Pfdefn}
|Pf(\nu)|^2= \det B_{\nu, P}, \:\:\:\: \txt{ for all } \nu\in \mathcal{U}.
\ee
We define 
\bes
V_P= \txt{span }\{X_i^*: i\in P\}, \txt{ and } V_Q = \txt{span} \{X_i^*: i\in Q\}.
\ees

Let $d\nu$ be the Lebesgue measure on $V_Q$ normalized so that the unit cube spanned by $\{X_i^*: i\in Q\}$ has volume $1$.
Then $\g^*=V_Q\oplus V_P$ and $\mathcal W=\mathcal U\cap V_Q$ is a cross-section for the coadjoint orbits through points in $\mathcal U$.

For $\phi\in L^1(G)$ we define its operator-valued Fourier transform by
\bes
\pi(\phi)= \int_{G}\phi(g) \pi(g^{-1}) ~dg, \:\:\:\: \txt{ for } \pi \in \widehat G,
\ees
where $dg$ denotes the Haar measure of $G$ and $\widehat G$ is the unitary dual of $G$. If $\phi\in L^1(G)\cap L^2(G)$ then $\pi(\phi)$ is a Hilbert-Schmidt operator and $\|\pi(\phi)\|_{HS}$ will denote the Hilbert-Schmidt norm of $\pi(\phi)$. If $d\nu$ denotes the element of Lebesgue measure on $\mathcal W$, then $\mu$ is the Plancherel measure for $\hat G$, where $d\mu= |Pf(\nu)|d\nu$. The Plancherel formula is thus given by
\bes
\|\phi\|_2^2=\int_G |\phi(g)|^2~dg= \int_{\mathcal W}\|\pi_\nu(\phi)\|_{HS}^2~d\mu(\nu), \:\:\:\: \textit{ for all } \phi\in L^1(G)\cap L^2(G).
\ees

We shall now state and prove the following analogue of Theorem \ref{introthm}, (c) on a connected, simply connected nilpotent Lie group $G$.

\begin{thm} \label{nilpotent}
Let $\psi: [0,\infty) \ra [0,\infty)$ be a locally integrable  function and 
\bes
I= \int_{1}^{\infty} \frac{\psi(t)}{t^2}~dt.
\ees
Suppose $f\in C_c(G)$ is such that 
\be \label{nilest} 
\int_{\mathcal W} \|\pi_{\nu}(f)\|^2_{HS} ~ e^{2\psi(|\nu_1|)} ~ |Pf(\nu)| ~ d\nu < \infty,
\ee 
where $\nu_1= \nu(X_1)$ and $X_1$ is given by (\ref{JHB}). If $I$ is infinite then $f$ vanishes identically on $G$. Conversely, if $\psi$ is nondecreasing and $I$ is finite, then there exists a nontrivial $f\in C_c(G)$ such that (\ref{nilest}) holds.
\end{thm} 

To prove the theorem we need some more notation. If $r$ is the dimension of coadjoint orbits $O_\nu$, $\nu\in \mathcal{U}$, then $Q$ has $d-r$ elements and so $V_Q$ can be identified with $\R^{d-r}$. In an abuse of notation we write $V_Q= \R X_1^*\oplus \R ^{d-r-1}$ and let 
\bes
p^*:V_Q\rightarrow \R X_1^*,\:\:\:\: \nu\mapsto \nu_1 X_1^* 
\ees
denote the canonical projection. As $\mathcal W$ is a Zariski open set in $V_Q$, $p^*(\mathcal W)=\mathcal O$ is also a nonempty Zariski open subset of $\R$. In particular, this is a subset in $\R$ with full Lebesgue measure. So, it will be convenient to write elements $\nu\in \mathcal W$ as $(\nu_1, \nu')$, where $\nu_1\in \mathcal O$ and $\nu'$ is in the set $\mathcal W_{\nu_1}$ defined by 
\bes
\mathcal W_{\nu_1}= \{\nu'\in \R^{n-d-1}: (\nu_1, \nu')\in \mathcal W\}.
\ees
It turns out that $\mathcal W_{\nu_1}$ is also a Zariski open subset of $\R^{n-d-1}$, for each fixed $\nu_1\in \mathcal O$. Given $f\in C_c(G)$ and the Jordan-H\"older basis $\{X_1, \cdots, X_d\}$ given in (\ref{JHB}) we define for each $y\in \R^{d-1}$ 
\bes
f_y(t)=f\left(\exp\left(tX_1+\sum_{j=2}^{d}{y_jX_j}\right)\right), \:\:\:\: \textit{ for } t\in \R,
\ees
and $f_y^*(t)=\overline{f_y(-t)}$. We now consider the function
\bes
g(t)=\int_{\mathbb{R}^{d-1}}{(f_y*f_y^*)(t)dy}, \:\:\:\: \textit{for } t\in \R.
\ees
Since $f\in C_c(G)$ it follows that  $g\in C_c(\R)$. In particular, $g \in L^1(\R)$. We now have the following lemma.

\begin{lem}(\cite{BTH}, Lemma 3.2) \label{lemnil}
For all $\nu_1\in \mathcal O$,
\bes
\widehat g(\nu_1)= \int_{\mathcal W_{\nu_1}} \|\pi_{\nu}(f)\|_{HS}^2 ~ |Pf(\nu)| ~ d\nu',
\ees
where $\widehat g$ is the one dimensional Fourier transform of $g$.
\end{lem} 

We now present the proof of the theorem.

{\em Proof of Theorem \ref{nilpotent}.}
It follows from the definition of $g$ that
 
\bes
\widehat g(\xi)= \int_{\R^{d-1}}{|\widehat{f_y}(\xi)|^2dy}, \:\:\:\: \xi\in \R,
\ees
and consequently $\widehat g \geq 0$ with
\bes
\int_{\R}\widehat g(\xi)=\int_{\R^d}{|f_y(\xi)|^2dy~d\xi}= \|f\|_2^2.
\ees
Therefore the function $f$ is identically zero on $G$ if and only if $g$ is zero on $\R$. We now complete the proof by showing that $g$ vanishes identically. In order to do this we shall apply Theorem \ref{paleywiener} to the function $g$. By Lemma \ref{lemnil}, it follows that
\beas
\int_{\R} \widehat g(\nu_1) ~ e^{\psi(|\nu_1|)} ~ d\nu_1 &=& \int_{\mathcal O} \int_{\mathcal W_{\nu_1}} ~ \|\pi_{\nu}(f)\|_{HS}^2 ~ e^{\psi(|\nu_1|)} ~ |Pf(\nu)| ~ d\nu' ~ d\nu_1\nonumber\\
& = & \int_{\mathcal W} \|\pi_{\nu}(f)\|_{HS}^2 ~ e^{\psi(|\nu_1|)} ~ |Pf(\nu)| ~ d\nu.
\eeas
By the hypothesis (\ref{nilest}) the right-hand side integral is finite. Therefore by Theorem \ref{paleywiener} we conclude that $g$ is identically zero on $\R$ and hence $f$ vanishes identically.

Conversely, if $\psi$ is non decreasing and the integral $I$ is finite then by Theorem \ref{introthm}, (b), for a given $\delta$ positive there exists $g\in C_c(\R)$ with $\txt{supp } g \subset [-\delta, \delta]$ satisfying the estimate 
\be \label{nilgest}
|\widehat g(y)| \leq Ce^{-\psi(|y|)}, \:\:\:\:  y\in \R.
\ee
Let $Z=\exp(\mathbb{R}X_1)$ denote a central
subgroup of $G$ and we identify $Z$ with $\mathbb{R}$. Let $h\in C_c(G)$ with $\txt{supp } h \subset V$ where $V$ is a symmetric neighbourhood of the identity element in $G$. We now define a function $f$ on $G$ by 
\bes f(x)=\int_{Z} g(t) ~ h(t^{-1}x) ~ dt, \:\:\:\: x\in G.
\ees
Clearly, $f$ is continuous. Moreover, $f\in C_c(G)$, which follows from the fact that
\bes
\|t^{-1}x\| \geq \|x\|- \|t\|
\ees
as $t\in Z$. It can be shown (\cite{KK}, P. 490)  that 
\be \label{kkreln}
\|\pi_\nu(f)\|^2_{HS} = |\widehat g(\nu_1)|^2 ~ \|\pi_\nu(h)\|^2_{HS}, \:\:\:\: \txt{ for all } \nu \in \mathcal W.  
\ee
Therefore, from (\ref{nilgest}) and (\ref{kkreln}) it follows that 
\beas
& & \int_{\mathcal W} \|\pi_\nu(f)\|^2_{HS} ~ e^{2\psi(|\nu_1|)} ~ |Pf(\nu)| ~ d\nu \\
&=& \int_{\mathcal W} |\widehat g(\nu_1)|^2 ~ \|\pi_\nu(h)\|^2_{HS} ~ e^{2\psi(|\nu_1|)} ~ |Pf(\nu)| ~ d\nu \\
& \leq & C \int_{\mathcal W} \|\pi_\nu(h)\|^2_{HS} ~ |Pf(\nu)| ~ d\nu \\
& = & C ~ \|h\|_{L^2(G)}  < \infty.
\eeas  
This completes the proof. \qed

If the function $\psi$ is assumed to be nondecreasing then we have the following corollary.
\begin{cor}
Let $\psi: [0,\infty) \ra [0,\infty)$ be a nondecreasing function and 
\bes
I= \int_{1}^{\infty} \frac{\psi(t)}{t^2}~dt.
\ees
Suppose $f\in C_c(G)$ is such that 
\bes 
\int_{\mathcal W} \|\pi_{\nu}(f)\|^2_{HS} ~ e^{2\psi(\|\nu\|)} ~ |Pf(\nu)| ~ d\nu < \infty.
\ees 
If the integral $I$ is infinite then $f$ vanishes identically on $G$. 
\end{cor}
\begin{proof}
Since $\psi$ is nondecreasing $\psi(\|\nu\|) \geq \psi(|\nu_1|)$. Therefore, the result follows from Theorem \ref{nilpotent}.
\end{proof}

\textbf{Acknowledgement.} We would like to thank Swagato K. Ray and Suparna Sen for the many useful discussions during the course of this work.


\begin{thebibliography}{99}

\bibitem{BB} Baklouti, A; Ben Salah, Nour . \textit{On theorems of Beurling and Cowling-Price for certain nilpotent Lie groups}, Bull. Sci. Math. 132 (2008), no. 6, 529-550. MR2445579 (2009f:22007)

\bibitem{BTH} Baklouti, A; Thangavelu, S. \textit{Variants of Miyachi's theorem for nilpotent Lie groups}, J. Aust. Math. Soc. 88 (2010), no. 1, 1-17. MR2770922 (2012a:22012)

\bibitem{BRS} Bhowmik, M.; Ray, S. K.; Sen, S. \textit{Around theorems of Ingham-type regarding decay of Fourier transform on $\R^n$, $\T^n$ and two step nilpotent Lie Groups}, to appear in Bull. Sci. Math.

\bibitem{B} Bochner, S. \textit{Quasi-analytic functions, Laplace operator, positive kernels} Ann. of Math. (2) 51 (1950), 68-91. MR0032708 (11,334g)

\bibitem{CG} Corwin,L. J.; Greenleaf, F. P. \textit{Representations of nilpotent Lie groups and their applications} Cambridge University Press, Cambridge, 1990. MR1070979 (92b:22007)

\bibitem{Hi} Hirschman, I. I. \textit{On the behaviour of Fourier transforms at infinity and on quasi-analytic classes of functions} Amer. J. Math. 72 (1950), 200-213. MR0032816 (11,350f)

\bibitem{I} Ingham, A. E. \textit{A Note on Fourier Transforms} J. London Math. Soc. 9 (1934) no. 1 , 29-32. MR1574706

\bibitem{KK} Kaniuth, E.; Kumar, A. \textit{Hardy's theorem for simply connected nilpotent Lie groups} Math. Proc. Cambridge Philos. Soc. 131 (2001), no. 3, 487-494. MR1866390 (2002j:22007) 

\bibitem{L1} Levinson, N. \textit{Gap and Density Theorems} American Mathematical Society Colloquium Publications, v. 26. American Mathematical Society, New York, 1940. MR0003208 (2,180d)

\bibitem{PW} Paley, R. E. A. C.; Wiener, N. \textit{Fourier transforms in the complex domain (Reprint of the 1934 original)} American Mathematical Society Colloquium Publications, 19. American Mathematical Society, Providence, RI, 1987. MR1451142 (98a:01023)

\bibitem{PW1} Paley, R. E. A. C.; Wiener, N. \textit{Notes on the theory and application of Fourier transforms. I, II.} Trans. Amer. Math. Soc. 35 (1933), no. 2, 348–355. MR1501688

\end{thebibliography}
\end{document}